\documentclass[12pt,reqno]{amsart}
\usepackage{amsmath,amsfonts,amssymb,amsthm,amsxtra,amscd}
\theoremstyle{plain}
\newtheorem{theorem}{Theorem}

\newtheorem{lemma}[theorem]{Lemma}

\begin{document}

\title[Tyurin's babylonian tower theorem]{A simple proof of Tyurin's 
babylonian tower theorem}

\author[I. Coand\u{a}]{Iustin Coand\u{a}}
\address{Institute of Mathematics of the Romanian Academy, 
         P. O. Box 1-764, RO-014700, Bucharest, Romania}
\email{Iustin.Coanda@imar.ro}

\subjclass[2000]{Primary: 14F05; Secondary: 14J60; 14B10}

\keywords{Vector bundle; Projective scheme; Babylonian tower}

\begin{abstract}
Using the method of Coand\u{a} and Trautmann (2006), we give a simple proof 
of the following theorem due to Tyurin (1976) in the smooth case: if a 
vector bundle $E$ on a $c$-codimensional locally Cohen-Macaulay closed 
subscheme $X$ of the projective space ${\mathbb P}^n$ extends to a vector 
bundle $F$ on a similar closed subscheme $Y$ of ${\mathbb P}^N$, for every 
$N > n$, then $E$ is the restriction to $X$ of a direct sum of line bundles 
on ${\mathbb P}^n$. Using the same method, we also provide a proof of the 
Babylonian tower theorem for locally complete intersection subschemes 
of projective spaces.     
\end{abstract}

\maketitle

Let ${\mathbb P}^n$ be the projective $n$-space over an algebraically closed 
field $k$ of arbitrary characteristic and $S = k[X_1,\ldots ,X_n]$ its 
projective coordinate ring. For $m > 0$, embed ${\mathbb P}^n$ into the 
projective $(n+m)$-space ${\mathbb P}^{n+m}$ with coordinate ring 
$R = k[X_1,\ldots ,X_{n+m}]$ as the linear subspace $L$ of equations 
$X_{n+1} = \cdots = X_{n+m} = 0$. One says that a coherent sheaf $\mathcal F$ 
on ${\mathbb P}^n$ {\it extends} to a coherent sheaf $\mathcal G$ on 
${\mathbb P}^{n+m}$ if ${\mathcal G}\, \vert \, {\mathbb P}^n \simeq 
\mathcal F$ and ${\mathcal Tor}_i^{{\mathcal O}_{{\mathbb P}^{n+m}}}({\mathcal G},
{\mathcal O}_L) = 0$, $\forall i > 0$. Since $L$ is defined locally in 
${\mathbb P}^{n+m}$ by a regular sequence, the later condition is equivalent 
to ${\mathcal Tor}_1^{{\mathcal O}_{{\mathbb P}^{n+m}}}({\mathcal G},{\mathcal O}_L) 
= 0$ (see Matsumura (1986, Thm. 16.5)). A closed subscheme $X$ of 
${\mathbb P}^n$ {\it extends} to a closed subscheme $Y$ of 
${\mathbb P}^{n+m}$ if the structure sheaf ${\mathcal O}_X$ extends to 
${\mathcal O}_Y$. It is easy to see that in this case the ideal sheaf 
${\mathcal I}_{X,{\mathbb P}^n}$ extends to ${\mathcal I}_{Y,{\mathbb P}^{n+m}}$. 

In this note, using the method of Coand\u{a} and Trautmann (2006), we shall 
provide simple, elementary proofs of the following two results: 

\begin{theorem}
Let $E$ be a vector bundle (= locally free sheaf) on a locally 
Cohen-Macaulay closed subscheme $X$ of ${\mathbb P}^n$, of pure 
codimension $c$. If, for every $m > 0$, $E$ extends to a vector bundle $F$ 
on a locally Cohen-Macaulay closed subscheme $Y$ of ${\mathbb P}^{n+m}$, of 
pure codimension $c$, then $E$ is isomorphic to a direct sum of line bundles 
of the form ${\mathcal O}_X(a)$, $a \in {\mathbb Z}$. 
\end {theorem}

\begin{theorem}
Let $X$ be a locally complete intersection closed subscheme of 
${\mathbb P}^n$, of pure codimension $c$. If, for every $m > 0$, $X$ extends 
to a locally complete intersection closed subscheme $Y$ of 
${\mathbb P}^{n+m}$, of pure codimension $c$, then $X$ is a complete 
intersection. 
\end{theorem}

Notice that, under the hypothesis of Theorem 1, $E$ extends to $F$ if and 
only if $F \, \vert \, {\mathbb P}^n \simeq E$, and, under the hypothesis of 
Theorem 2, $X$ extends to $Y$ if and only if $Y \cap {\mathbb P}^n = X$ 
(as schemes). 

In the case where $X$ and $Y$ are assumed to be smooth, Theorem 1 is due to 
Tyurin (1976), and Theorem 2 is due to Barth and Van de Ven (1974), 
Barth (1975) and Tyurin (1976). The more general version of Theorem 2 stated 
above is due to Flenner (1985). 

The proofs of these theorems are based on the next three lemmas. The first 
two of them are very elementary and the proof of the third one uses the idea 
of Coand\u{a} and Trautmann (2006). Before stating and proving them, we 
recall the following notation: for $i \geq 0$, 
$\text{H}^i_{\ast}(\mathcal F)$ denotes the graded $S$-module 
$\bigoplus_{d\in {\mathbb Z}}\text{H}^i({\mathcal F}(d))$, and, for $Z$ closed 
subscheme of ${\mathbb P}^{n+m}$, ${\mathcal G}_Z$ denotes the sheaf 
${\mathcal G}\otimes_{{\mathcal O}_{{\mathbb P}^{n+m}}}{\mathcal O}_Z$. 

\begin{lemma}
Assume that a coherent sheaf $\mathcal A$ on ${\mathbb P}^N$ extends to a 
coherent sheaf $\mathcal B$ on ${\mathbb P}^{N+1}$ with the property that 
${\fam0 H}^0({\mathcal B}(-t)) = 0$ for $t >> 0$. Let $\cdots \rightarrow 
G^{-1} \rightarrow G^0 \rightarrow {\fam0 H}^0_{\ast}(\mathcal B) 
\rightarrow 0$ be a graded minimal free resolution of 
${\fam0 H}^0_{\ast}(\mathcal B)$ over $k[X_0,\ldots ,X_{N+1}]$. If 
${\fam0 H}^0_{\ast}(\mathcal B) \rightarrow {\fam0 H}^0_{\ast}(\mathcal A)$ is 
surjective then $G^{\bullet}/X_{N+1}G^{\bullet}$ is a minimal free resolution 
of ${\fam0 H}^0_{\ast}(\mathcal A)$ over $k[X_0,\ldots ,X_N]$. 
\end{lemma}

\begin{proof}
Using the exact sequence: 
\[
\begin{CD}
0 @>>> {\mathcal B}(-1) @>X_{N+1}>> 
{\mathcal B} @>>> {\mathcal A} @>>> 0 
\end{CD}
\]
one deduces, firstly, that $X_{N+1}$ is 
$\text{H}^0_{\ast}(\mathcal B)$-regular, hence $G^{\bullet}/X_{N+1}G^{\bullet}$ 
is a minimal free resolution of 
$\text{H}^0_{\ast}(\mathcal B)/X_{N+1}\text{H}^0_{\ast}(\mathcal B)$ over 
$k[X_0,\ldots ,X_N]$ and, then, that 
$\text{H}^0_{\ast}(\mathcal B)/X_{N+1}\text{H}^0_{\ast}(\mathcal B) \simeq 
\text{H}^0_{\ast}(\mathcal A)$.
\end{proof}

\begin{lemma}
Assume that a coherent sheaf $\mathcal F$ on ${\mathbb P}^n$ extends to a 
coherent sheaf $\mathcal G$ on ${\mathbb P}^{n+m}$ with the property that 
${\fam0 H}^i({\mathcal G}(-t)) = 0$ for $t >> 0$, $i = 0, \ldots ,m-1$. 
Let $\cdots \rightarrow G^{-1} \rightarrow G^0 \rightarrow 
{\fam0 H}^0_{\ast}(\mathcal G) \rightarrow 0$ be a minimal free resolution of 
the graded $R$-module ${\fam0 H}^0_{\ast}(\mathcal G)$. If there exists an 
$(n+1)$-dimensional linear subspace $P$ of ${\mathbb P}^{n+m}$ containing 
$L = {\mathbb P}^n$ such that ${\fam0 H}^0_{\ast}({\mathcal G}_P) \rightarrow 
{\fam0 H}^0_{\ast}(\mathcal F)$ is surjective, then $G^{\bullet}\otimes_RS$ is 
a minimal free resolution of the graded $S$-module 
${\fam0 H}^0_{\ast}(\mathcal F)$. 
\end{lemma}

\begin{proof}
Consider a saturated flag of linear subspaces $P^{(0)} = L \subset P = P^{(1)} 
\subset \ldots \subset P^{(m)} = {\mathbb P}^{n+m}$. By decreasing induction 
on $j = m, \ldots ,1$ one shows easily that 
$\text{H}^i({\mathcal G}_{P^{(j)}}(-t)) = 0$ for $t >> 0$, $i = 0, \ldots ,
j-1$. Since $\text{H}^0_{\ast}({\mathcal G}_P) \rightarrow 
\text{H}^0_{\ast}(\mathcal F)$ is surjective and 
$\text{H}^1({\mathcal G}_P(t)) = 0$ for $t >> 0$, it follows that 
$\text{H}^1_{\ast}({\mathcal G}_P) = 0$. One shows now, by increasing induction 
on $j = 1, \ldots , m$, that $\text{H}^i_{\ast}({\mathcal G}_{P^{(j)}}) = 0$, 
$i = 1, \ldots , j$. But $\text{H}^1_{\ast}({\mathcal G}_{P^{(j)}}) = 0$ implies 
that $\text{H}^0_{\ast}({\mathcal G}_{P^{(j)}}) \rightarrow 
\text{H}^0_{\ast}({\mathcal G}_{P^{(j-1)}})$ is surjective, $j = 2 , \ldots , m$, 
and one finally applies Lemma 3.
\end{proof}

\begin{lemma}
Let $\mathcal F$ be a coherent sheaf on ${\mathbb P}^n$, $n \geq 2$, with the 
property that ${\fam0 H}^i({\mathcal F}(-t)) = 0$ for $t >> 0$, $i = 0, 1$. 
For $i \in {\mathbb Z}$, let ${\mu}_i$ denote the number of minimal generators 
of degree $i$ of the graded $S$-module ${\fam0 H}^0_{\ast}(\mathcal F)$. 
If, for some $m > \sum_{i>j}{\mu}_i\, {\fam0 h}^1({\mathcal F}(j))$, 
$\mathcal F$ extends to a coherent sheaf $\mathcal G$ on ${\mathbb P}^{n+m}$ 
with ${\fam0 H}^i({\mathcal G}(-t)) = 0$ for $t >> 0$, $i = 1, \ldots ,m$, 
then there exists an $(n+1)$-dimensional linear subspace $P$ of 
${\mathbb P}^{n+m}$ containing $L = {\mathbb P}^n$ such that 
${\fam0 H}^0_{\ast}({\mathcal G}_P) \rightarrow {\fam0 H}^0_{\ast}(\mathcal F)$ 
is surjective. 
\end{lemma}

\begin{proof}
We recall that $\text{h}^1({\mathcal F}(j))$ denotes 
$\text{dim}_k\text{H}^1({\mathcal F}(j))$. Now, for $i \geq 0$, let $L_i$ be 
the $i\, $th infinitesimal neighbourhood of $L$ in ${\mathbb P}^{n+m}$, defined 
by the ideal sheaf ${\mathcal I}^{i+1}_{L,{\mathbb P}^{n+m}}$. 
Let $L^{\prime}$ be the 
linear subspace of ${\mathbb P}^{n+m}$ of equations $X_0 = \ldots = X_n = 0$, 
with coordinate ring $S^{\prime} = k[X_{n+1},\ldots ,X_{n+m}]$. If $\pi : 
{\mathbb P}^{n+m} \setminus L^{\prime} \rightarrow L$ is the linear projection 
then $\pi \, \vert \, L_i \rightarrow L$ is a retract of the inclusion 
$L \hookrightarrow L_i$ and endows ${\mathcal O}_{L_i}$ with a structure of 
${\mathcal O}_L$-algebra. As an ${\mathcal O}_L$-module: 
\[
{\mathcal O}_{L_i} \simeq {\mathcal O}_L \oplus {\mathcal O}_L(-1)
\otimes_kS^{\prime}_1 \oplus \cdots \oplus {\mathcal O}_L(-i)\otimes_k 
S^{\prime}_i\, .
\]
Moreover, if $J \subset S^{\prime}$ is a homogeneous ideal then: 
\[
{\mathcal O}_L(-1)\otimes_k J_1 
\oplus \cdots \oplus {\mathcal O}_L(-i)\otimes_k J_i
\]
is an ideal sheaf of ${\mathcal O}_{L_i}$, hence defines a closed subscheme 
$Y_i$ of ${\mathbb P}^{n+m}$ with $L \subseteq Y_i \subseteq L_i$. Using the 
exact sequence:
\[
0 \longrightarrow {\mathcal O}_L(-i-1)\otimes_k(S^{\prime}_{i+1}/J_{i+1}) 
\longrightarrow {\mathcal O}_{Y_{i+1}} \longrightarrow 
{\mathcal O}_{Y_i} \longrightarrow 0
\]
one deduces easily, by induction on $i$, that 
${\mathcal Tor}_1^{{\mathcal O}_{{\mathbb P}^{n+m}}}({\mathcal G},{\mathcal O}_{Y_i}) 
= 0$, $\forall i \geq 0$. It follows that, tensorizing the above exact 
sequence by $\mathcal G$, one gets an exact sequence: 
\[
0 \longrightarrow {\mathcal F}(-i-1)\otimes_k(S^{\prime}_{i+1}/J_{i+1}) 
\longrightarrow {\mathcal G}_{Y_{i+1}} \longrightarrow 
{\mathcal G}_{Y_i} \longrightarrow 0\, .
\]

Consider, now, a section $s \in \text{H}^0(\mathcal F)$. Using an argument 
similar to that used in the proof of Coand\u{a} (2010, Lemma 5), one can show 
that there exists a homogeneous ideal $J(s) \subset S^{\prime}$, with at most 
$\text{h}^1({\mathcal F}(-j))$ minimal generators in degree $j$, $\forall j 
\geq 1$, with the property that if $Y_i(s)$ is the closed subscheme of 
$L_i$ defined by the ideal sheaf:
\[
{\mathcal O}_L(-1)\otimes_k J(s)_1 
\oplus \cdots \oplus {\mathcal O}_L(-i)\otimes_k J(s)_i
\]
then $s$ can be lifted to a global section of ${\mathcal G}_{Y_i(s)}$, 
$\forall i \geq 1$. Choosing, next, a minimal system of generators of the 
graded $S$-module $\text{H}^0_{\ast}(\mathcal F)$ one deduces the existence 
of an ideal $J \subset S^{\prime}$ generated by at most 
$\sum_{i>j}{\mu}_i\, {\fam0 h}^1({\mathcal F}(j))$ homogeneous elements such 
that $\text{H}^0_{\ast}({\mathcal G}_{Y_i}) \rightarrow \text{H}^0_{\ast}
(\mathcal F)$ is surjective, $\forall i \geq 1$. 

Since $m > \sum_{i>j}{\mu}_i\, {\fam0 h}^1({\mathcal F}(j))$, there exists a 
point $p \in L^{\prime} \simeq {\mathbb P}^{m-1}$ such that all the elements 
of $J$ vanish at $p$. Let $P \subset {\mathbb P}^{n+m}$ be the linear span of 
$L$ and $p$. One has $P\cap L_i \subseteq Y_i$ hence 
$\text{H}^0_{\ast}({\mathcal G}_{P\cap L_i}) \rightarrow \text{H}^0_{\ast}
(\mathcal F)$ is surjective, $\forall i \geq 1$. But $P\cap L_i$ is the 
$i\, $th infinitesimal neighbourhood in $P$ of the hyperplane $L$ of $P$. 
Tensorizing by ${\mathcal G}(d)$, for a {\it fixed} $d \in {\mathbb Z}$, 
the exact sequence: 
\[
0 \longrightarrow {\mathcal O}_P(-i-1) \longrightarrow {\mathcal O}_P 
\longrightarrow {\mathcal O}_{P\cap L_i} \longrightarrow 0
\]
one gets an exact sequence: 
\[
\text{H}^0({\mathcal G}_P(d)) \longrightarrow 
\text{H}^0({\mathcal G}_{P\cap L_i}(d)) \longrightarrow 
\text{H}^1({\mathcal G}_P(d-i-1))\, .
\]
Since $\text{H}^i({\mathcal G}(-t)) = 0$ for $t >> 0$, $i = 1, \ldots ,m$, 
one deduces, as in the proof of Lemma 4, that 
$\text{H}^1({\mathcal G}_P(d-i-1)) = 0$, for $i >> 0$, hence 
$\text{H}^0({\mathcal G}_P(d)) \rightarrow 
\text{H}^0({\mathcal G}_{P\cap L_i}(d))$ is surjective for $i >> 0$. This 
implies that 
$\text{H}^0_{\ast}({\mathcal G}_P) \rightarrow \text{H}^0_{\ast}(\mathcal F)$ 
is surjective.
\end{proof}

\begin{proof}[Proof of Theorem 1]
Let ${\mu}_i$ be the number of minimal generators of degree $i$ of the 
graded $S$-module $\text{H}^0_{\ast}(E)$ and consider a minimal free 
$S$-resolution $\cdots \rightarrow L^{-1} \rightarrow L^0 \rightarrow 
\text{H}^0_{\ast}(E) \rightarrow 0$ with, of course, 
$L^0 \simeq \bigoplus_{i\in {\mathbb Z}}S(-i)^{{\mu}_i}$. Let $r_i := 
\text{rk}\, L^{-i}$, $i \geq 0$. We want to show that $r_0 = \text{rk}\, E =: 
r$. 

Now, from the hypothesis, $\text{H}^i(E(-t)) = 0$ for $t >> 0$, $i = 0, 
\ldots , n-c-1$ and $\text{H}^i(F(-t)) = 0$ for $t >> 0$, $i = 0, \ldots , 
n+m-c-1$. We may assume that $n-c \geq 2$. Let $\cdots \rightarrow K^{-1} 
\rightarrow K^0 \rightarrow \text{H}^0_{\ast}(F) \rightarrow 0$ be a minimal 
free resolution of the graded $R$-module $\text{H}^0_{\ast}(F)$. It follows 
from Lemma 4 and Lemma 5 that if 
$m > \sum_{i>j}{\mu}_i\, {\fam0 h}^1(E(j))$ then $K^{\bullet}\otimes_RS \simeq 
L^{\bullet}$. 

The sheafified morphism ${\widetilde K}^{-1} \rightarrow 
{\widetilde K}^0$ has constant corank $r$ along $Y$. If $r_0 > r$ then the 
$(r_0-r)\times (r_0-r)$ minors of the matrix defining this morphism vanish on 
a closed subscheme of ${\mathbb P}^{n+m}$ of codimension 
$\leq (r_1-(r_0-r)+1)(r_0-(r_0-r)+1) = (r_1-r_0+r+1)(r+1)$. If 
$\text{dim}\, Y \geq (r_1-r_0+r+1)(r+1)$, i.e., if $m \geq 
(r_1-r_0+r+1)(r+1) - n + c$, then one gets a contradiction.
\end{proof}

\begin{proof}[Proof of Theorem 2]
Let $\cdots \rightarrow F^{-1} \rightarrow F^0 \rightarrow I(X) \rightarrow 0$ 
be a minimal free resolution of the homogeneous ideal $I(X) := 
\text{H}^0_{\ast}({\mathcal I}_{X,{\mathbb P}^n})$ of $S$. One has 
$F^0 \simeq \bigoplus_{i\in {\mathbb Z}}S(-i)^{{\mu}_i}$, where ${\mu}_i$  
is the number of homogeneous minimal generators of degree $i$ of $I(X)$. 
Let $r_i := \text{rk}\, F^{-i}$, $i \geq 0$. We want to show that $r_0 = c$. 

Now, it follows from the hypothesis that 
$\text{H}^i({\mathcal I}_{X,{\mathbb P}^n}(-t)) = 0$ for $t >> 0$, $i = 0, 
\ldots n-c$, and that $\text{H}^i({\mathcal I}_{Y,{\mathbb P}^{n+m}}(-t)) = 0$ 
for $t >> 0$, $i = 0, \ldots , n+m-c$. We may assume that $n-c \geq 1$. 
Let $\cdots \rightarrow G^{-1} \rightarrow G^0 \rightarrow I(Y) \rightarrow 0$ 
be a minimal free resolution of the homogeneous ideal $I(Y)$ of $R$. From 
Lemma 4 and Lemma 5 one deduces that if 
$m > \sum_{i>j}{\mu}_i\text{h}^1({\mathcal I}_{X,{\mathbb P}^n}(j))$ then 
$G^{\bullet}\otimes_RS \simeq F^{\bullet}$. 

The sheafified morphism ${\widetilde G}^{-1} \rightarrow {\widetilde G}^0$ 
has constant corank $c$ along $Y$. If $r_0 > c$ then the $(r_0-c)\times 
(r_0-c)$ minors of the matrix defining this morphism vanish on a closed 
subscheme of ${\mathbb P}^{n+m}$ of codimension 
$\leq (r_1-(r_0-c)+1)(r_0-(r_0-c)+1) = (r_1-r_0+c+1)(c+1)$. If 
$\text{dim}\, Y \geq (r_1-r_0+c+1)(c+1)$, i.e., if 
$m \geq (r_1-r_0+c+1)(c+1) -n +c$, then one gets a contradiction.
\end{proof}

Using Lemma 4 and Lemma 5 one can also prove the following result, that 
answers a question the hypothesis of Theorem 1 might raise. 

\begin{theorem}
Let $X$ be a locally Cohen-Macaulay closed subscheme of ${\mathbb P}^n$, 
of pure codimension $c \geq 2$. If, for every $m > 0$, $X$ extends to a 
locally Cohen-Macaulay closed subscheme of ${\mathbb P}^{n+m}$, of pure 
codimension $c$, then $X$ is arithmetically Cohen-Macaulay. 
\end{theorem}

\begin{proof}
Consider, as in the above proof of Theorem 2, minimal free resolutions 
$\cdots \rightarrow F^{-1} \rightarrow F^0 \rightarrow I(X) \rightarrow 0$ and 
$\cdots \rightarrow G^{-1} \rightarrow G^0 \rightarrow I(Y) \rightarrow 0$. 
If $m$ is sufficiently large then $G^{\bullet}\otimes_RS \simeq F^{\bullet}$. 
Put $F^{-1} := S$ and $G^{-1} := R$. One deduces that the vector bundle 
$E := \text{Ker}({\widetilde F}^{-c+2} \rightarrow {\widetilde F}^{-c+3})$ on 
${\mathbb P}^n$ extends to the vector bundle 
$E^{(m)} := \text{Ker}({\widetilde G}^{-c+2} \rightarrow {\widetilde G}^{-c+3})$ 
on ${\mathbb P}^{n+m}$. By the Babylonian tower theorem for vector bundles on 
projective spaces of Barth and Van de Ven (1974), E. Sato (1977), (1978) and 
Tyurin (1976) (which is, of course, a particular case of Theorem 1) $E$ is a 
direct sum of line bundles on ${\mathbb P}^n$. It follows that 
$\text{H}^i_{\ast}({\mathcal I}_{X,{\mathbb P}^n}) = 0$, 
$i = 1, \ldots , n-c = \text{dim}\, X$, hence $X$ is arithmetically 
Cohen-Macaulay.
\end{proof}     

\section*{References}

\noindent
Barth, W. (1975). Submanifolds of low codimension in projective space. 
\emph{Proc. Int. Congr. Math.} Vancouver 1974. Vol 1:409--413.

\noindent
Barth, W., Van de Ven, A. (1974). A decomposability criterion for algebraic 
2-bundles on projective spaces. \emph{Invent. Math.} 25:91--106.

\noindent
Coand\u{a}, I. (2010). Infinitely stably extendable vector bundles on 
projective spaces. \emph{Arch. Math.} 94:539--545. 

\noindent
Coand\u{a}, I., Trautmann, G. (2006). The splitting criterion of Kempf and 
the Babylonian tower theorem. \emph{Comm. Algebra} 34:2485--2488. 

\noindent
Flenner, H. (1985). Babylonian tower theorems on the punctured spectrum. 
\emph{Math. Ann.} 271:153--160.

\noindent
Matsumura, H. (1986). \emph{Commutative Ring Theory}. Cambridge Studies in 
Advanced Mathematics 8: Cambridge University Press: Cambridge.

\noindent
Sato, E. (1977). On the decomposability of infinitely extendable vector 
bundles on projective spaces and Grassmann varieties. 
\emph{J. Math. Kyoto Univ.} 17:127--150. 

\noindent 
Sato, E. (1978). The decomposability of an infinitely extendable vector 
bundle on the projective space, II. In: \emph{International Symposium 
on Algebraic Geometry}. Kyoto University. Kinokuniya Book Store: Tokyo, 
pp. 663--672. 

\noindent
Tyurin, A.N. (1976). Finite dimensional vector bundles over infinite 
varieties. \emph{Math. USSR Izv.} 10:1187--1204.


\end{document}